\documentclass[a4paper,12pt]{amsart}

\usepackage{amssymb,amsmath,amsthm,latexsym}
\usepackage{amscd}
\newtheorem{Theorem}{Theorem}[section]

\newtheorem{Definition}[Theorem]{Definition}
\newtheorem{Proposition-Definition}[Theorem]{Proposition-Definition}
\newtheorem{Lemma}[Theorem]{Lemma}

\newtheorem{Example}[Theorem]{Example}

\newtheorem{Observation}[Theorem]{Observation}
\newtheorem{Question}[Theorem]{Question}

\def\Z{{\mathbb{{Z}}}}
\def\N{{\mathbb{N}}}
\def\Q{{\mathbb{Q}}}

\def\F{{\mathbb{F}}}

\def\m{{\mathfrak{m}}}

\def\rank{\mathop{\mathrm{rank}}\nolimits}

\def\max{\mathop{\mathrm{max}}\nolimits}

\def\deg{\mathop{\mathrm{deg}}\nolimits}

\def\Hom{\mathop{\mathrm{Hom}}\nolimits}

\def\initial{\mathop{\mathrm{in}}\nolimits}
\begin{document}
\title{One-dimensional rings of finite F-representation type}
\author{Takafumi Shibuta}
\address{
\begin{flushleft}
	\hspace{0.3cm} Department of Mathematics, Rikkyo University, Nishi-Ikebukuro, Tokyo 171-8501, Japan\\
	\hspace{0.3cm} JST, CREST, Sanbancho, Chiyoda-ku, Tokyo, 102-0075, Japan\\
\end{flushleft}
}
\email{shibuta@rikkyo.ac.jp}
\date{}
\baselineskip 15pt
\footskip = 32pt
\begin{abstract}
We prove that a complete local or graded one-dimensional domain of prime characteristic has finite F-representation type if 
its residue field is algebraically closed or finite, 
and present examples of a complete local or graded one-dimensional domain which does not have finite F-representation type 
with a perfect residue field. 
We also present some examples of higher dimensional rings of finite F-representation type. 
\end{abstract}
\maketitle
\section{Introduction}
Smith and Van den Bergh \cite{SV} introduced the notion of finite F-representation type 
as a characteristic $p$ analogue of the notion of finite Cohen-Macaulay representation type. 
Rings of finite F-representation type satisfy several nice properties. 
For example, Seibert \cite{S} proved that the Hilbert-Kunz multiplicities are rational numbers, 
Yao \cite{Y} proved that tight closure commutes with localization in such rings, 
and Takagi--Takahashi \cite{TT} proved that if $R$ is a Cohen-Macaulay ring of finite F-representation type with canonical module $\omega_R$, 
then $H^n_I(\omega_R)$ has only finitely many associated primes for any ideal $I$ of $R$ and any integer $n$. 
However, it is difficult to determine whether a given ring have finite F-representation type. 
In this paper, we prove the following results: 
Let $A$ be a complete or graded one-dimensional domain of prime characteristic with the residue field $k$. Then 
\begin{enumerate}
\item if $k$ is algebraically closed, then any finitely generated $A$-module has finite F-representation type, 
\item if $k$ is finite, then $A$ has finite F-representation type, 
\item there exist examples of rings which do not have finite F-representation type with $k$ perfect. 
\end{enumerate}

In the last section, we give some examples of finite F-representation type of dimension higher than one. 
There is a question posed by Brenner: 
\begin{Question}[Brenner]\upshape
Let $k$ be an algebraically closed field of characteristic $p$. 
Then does the ring $k[{x}, {y}, {z}]/({x}^2 + {y}^3 + {z}^7)$ have finite F-representation type? 
\end{Question}
We prove that $k[{x}, {y}, {z}]/({x}^2 + {y}^3 + {z}^7)$ has finite F-representation type if $p=2,3$ or $7$. 
\section{Rings of finite F-representation type}
Throughout this paper, all rings are Noetherian commutative rings of prime characteristic $p$. 
We denote by $\N=\{0,1,2,\dots\}$ the set of non-negative integers. 
If $R=\bigoplus_{i\ge 0}R_i$ is an $\N$-graded ring, we assume that $\gcd\{i\mid R_i\neq 0\}=1$. 
We denote by $R_+=\bigoplus_{i>0}R_i$ the unique homogeneous maximal ideal of $R$. 
 
The Frobenius map $F : R \to R$ is defined by sending $r$ to $r^p$ for all $r\in R$. 
For any $R$-module $M$, we denote by $^eM$ the module $M$ with its $R$-module structure pulled back via the $e$-times 
iterated Frobenius map $F^e:r\mapsto r^{p^e}$, 
that is, $^eM$ is the same as $M$ as an abelian group, 
but its $R$-module structure is determined by $r\cdot m := r^{p^e}m$ for $r \in R$ and $m \in M$. 
If $^1R$ is a finitely generated $R$-module (or equivalently $^eR$ is a finitely generated
$R$-module for every $e\ge 0$), we say that $R$ is F-finite. 
In general, if $^1M$ is a finitely generated $R$-module, we say that $M$ is F-finite. 
Remark that when $R$ is reduced, $^eR$ is isomorphic to $R^{1/q}$ where $q=p^e$. 
If $R$ and $M$ are $\Z$-graded, then $^eM$ carries a $\Q$-graded $R$-module structure: 
We grade $^eM$ by putting $[^eM]_\alpha = [M]_{p^e\alpha}$ if $\alpha\in \frac{1}{p^e}\Z$, otherwise $[M]_{\alpha} = 0$. 
For $\Q$-graded modules $M$ and $N$ and a rational number $r$, 
we say that a homomorphism $\phi:M\to N$ is {\it homogeneous (of degree $r$)} if $\psi(M_s)\subset N_{s+r}$ for all $s\in \Q$. 
We denote by $\Hom_r(M,N)$ the group of homogeneous homomorphisms of degree $r$, 
and set $\underline{\Hom}_R(M,N)=\bigoplus_{r\in\Q}\Hom_r(M,N)$. 

Let $I$ be an ideal of $R$. Then for any $q = p^e$, we use $I^{[q]}$ to denote the ideal 
generated by $\{x^q \mid  x\in I\}$. 
For any $R$-module $M$, it is easy to see that 
$(R/I)\otimes_R {}^eM\cong {}^eM/(I\cdot {}^eM) \cong {}^e(M/I^{[q]}M)$. 
Since the functor $^e(-)$ is a exact functor, and $I$ and $I^{[q]}$ have the same radical, 
we have $^eH_I(M)\cong{}^eH_{I^{[q]}}(M)\cong H_I(^eM)$.  
\begin{Definition}\upshape
\mbox{}
\begin{enumerate}
\item Let $R$ be a ring of prime characteristic $p$, and $M$ a finitely generated $R$-module.
We say that $M$ has {\it finite F-representation type} by finitely generated $R$-modules $M_1,\cdots,M_s$ 
if for every $e\in \N$, the $R$-module $^eM$ is isomorphic to 
a finite direct sum of the $R$-modules $M_1,\cdots,M_s$, that is, 
there exist nonnegative integers $n_{e,1},\cdots, n_{e,s}$ such that 
\[
^eM \cong \bigoplus_{i=1}^{s} M_i^{\oplus n_{e,i}}. 
\]
We say that a ring $R$ has finite F-representation type if $R$ has finite F-representation type as an $R$-module. 
\item
Let $R=\bigoplus_{n\ge 0}R_n$ be a Noetherian graded ring of prime characteristic $p$, and 
$M$ a finitely generated graded $R$-module. 
We say that $M$ has {\it finite graded F-representation type} by finitely generated $\Q$-graded $R$-modules $M_1,\dots,M_s$ 
if for every $e\in \N$, the $\Q$-graded $R$-module $^eM$
is isomorphic to a finite direct sum of the $\Q$-graded $R$-modules $M_1,\dots,M_s$ up to shift of grading,
that is, there exist non-negative integers $n_{ei}$ for $1 \le i \le s$, and rational numbers ${a}^{(e)}_{ij}$ 
for $1 \le j \le n_{ei}$ such that there exists a $\Q$-homogeneous isomorphism 
\[
^eM \cong \bigoplus_{i=1}^s \bigoplus_{j=1}^{n_{ei}}M_i({a}^{(e)}_{ij}),
\]
where $M(a)$ stands for the module obtained from $\Q$-graded module $M$ by the shift of grading by $a\in \Q$; 
$[M(a)]_b := M_{a+b}$ for $b\in \Q$. 
We say that a graded ring $R$ has finite graded F-representation type 
if $R$ has finite graded F-representation type as a graded $R$-module. 
\end{enumerate}
\end{Definition}
Note that if $M$ has finite F-representation type, then $M$ is F-finite. 
In this paper, we mainly investigate the cases where $R$ is a complete local Noetherian ring or 
$\N$-graded ring $R = \bigoplus_{i\ge 0} R_i$ with $R_0=k$ a field. 
Remark that the Krull-Schmidt theorem holds in these cases. 
\begin{Example}\upshape\label{exffrt}
\mbox{}
\begin{enumerate}
\renewcommand{\theenumi}{\roman{enumi}}
\renewcommand{\labelenumi}{\rm{(\theenumi)}}
\item 
Direct sums, localizations, or completions of modules of finite F-representation type also have finite F-representation type. 
\item
Let $R$ be an F-finite regular local ring or a polynomial ring $k[t_1,\dots,t_r]$ over a
field $k$ of characteristic $p>0$ such that $[k:k^p]<\infty$. 
Then $R$ has finite F-representation type. 
\item
Let $R$ be Cohen-Macaulay local (resp. graded) ring with finite (resp. graded) Cohen-Macaulay representation type, 
that is, there are finitely many isomorphism classes of indecomposable (resp. graded) maximal Cohen-Macaulay $R$-modules. 
Then every finitely generated (resp. graded)  maximal Cohen-Macaulay $R$-modules has finite F-representation type. 
\item
Let $(R,\m,k)$ be an F-finite local ring (resp. $\N$-graded ring with $k=R/R_+\cong R_0$), and $M$ an $R$-module of finite length $\ell(M)$. Then $M$ has finite F-representation type; 
$^eM\cong k^{\ell(M)a^e} \mbox{~~for~~sufficiently~~large~~} q=p^e$ where $a=[k:k^p]$. 
In particular, Artinian F-finite local rings have finite F-representation type. 
\item
Let $R\hookrightarrow S$ be a finite local homomorphism of Noetherian local rings of prime characteristic $p$ such that 
$R$ is an $R$-module direct summand of $S$. If $S$ has finite F-representation type, so does $R$. 
\item
(\cite{SV}, Proposition 3.1.6]) Let $R =\bigoplus_{i\ge 0}R_i \subset S =\bigoplus_{i\ge 0}S_i$ be a Noetherian 
$\N$-graded rings with $R_0$ and $S_0$ fields of characteristic $p>0$ such that $R$ is an $R$-module direct summand of $S$. 
Assume in addition that $[S_0 : R_0] < \infty$. 
If $S$ has finite graded F-representation type, so does $R$. In particular, 
normal semigroup rings and rings of invariants of linearly reductive groups have finite graded F-representation type. 
\end{enumerate}
\end{Example}
\section{One-dimensional case}
In this section, we investigate whether one-dimensional complete local or $\N$-graded domains have finite F-representation type. 
\begin{Theorem}\label{one dim}
Let $(A,\m,k)$ be a one-dimensional complete local domain (resp. an $\N$-graded domain $A=\bigoplus_{i\ge 0}A_i$, $\m=\bigoplus_{i>0}A_i$ 
and $A_0\cong A/\m=k$) of prime characteristic $p$. 
Let $M$ be a finitely generated (resp. graded) $A$-module. 
Assume that $k$ is an algebraically closed field. 
Then for sufficiently large $e\gg 0$, 
\[
^eM\cong B^{\oplus rq} \oplus k^\ell \quad \mbox{$(q=p^e)$}
\]
where $B$ is the integral closure of $A$, $r$ is the rank of $M$, and $\ell$ is the length of $H_\m^0(M)$. 
In particular, $M$ has finite F-representation type. 
\end{Theorem}
\begin{proof}
In the case where $A$ is a complete local domain, $B$ is isomorphic to a formal power series ring $k[[t]]$. 
For $f\in B$, we set $v_B(f)=\min\{i\mid f\in t^iB\}$. 
Let $H=\{v_B(f)\mid f\in A\}$, and $c(H)=\min\{j\mid i\in H \mbox{~~if~~} i\ge j\}$. 
Since $\N\backslash H$ is a finite set and $A$ is complete, it follows that $t^i\in A$ for all $i\ge c(H)$. 
Let $n=\min \{n\mid \m^n H_\m^0(M)=0\}$ and $q=p^e\ge \max\{c(H),n\}$. 
Since $B^q\subset A$, $^eM$ has a $B$-module structure. 
Thus $^eM\cong B^{\oplus rq}\oplus H^0_\m(^eM)$ because $B$ is a principal ideal domain and $\rank(^eM)=rq$. 
Since $H^0_\m(^eM)\cong ^eH^0_\m(M)\cong k^\ell$, we conclude the assertion. 

In the case where $A$ is an $\N$-graded ring, $B$ is isomorphic to a polynomial ring $k[t]$. 
Since $A=k[t^{n_1},\dots,t^{n_r}]$ for some $n_i\in \N$ with $\gcd(n_1,\dots,n_r)=1$, 
we can prove the assertion similarly to the complete case. 
\end{proof}
The assumption that $k$ is algebraic closed is essential for this theorem. 
Let $A=\bigoplus_{i\ge 0}A_i$ be a one-dimensional $\N$-graded domain with $A_0=k$ a perfect field. 
Then $^eA\cong A^{1/q}$ has rank $[k:k^q]q=q$, and is decomposed to $A$-modules of rank one by degree; 
$A^{1/q}=\bigoplus_{i=0}^{q-1} M^{(e)}_i$ where 
\[
M^{(e)}_i=\bigoplus_{j\equiv i\!\!\mod q} [A^{\frac{1}{q}}]_{\frac{j}{q}} 
\]
where $[A^{\frac{1}{q}}]_{\frac{j}{q}}$ is the degree ${\frac{j}{q}}$ component of $A^{\frac{1}{q}}$. 
Let $B$ be the integral closure of $A$. 
Then it follows that $B$ is isomorphic to a polynomial ring $K[t]$ with $\deg t=1$ for some $K$, a finite degree extension of $k$. 
Note that $K$ is also a perfect field. 
We can write $A=k[\alpha_1t^{n_1},\dots,\alpha_rt^{n_r}]$ for some $n_1,\dots,n_r\in \N$ and $\alpha_1,\dots,\alpha_r\in K$. 
For $i\in \N$, we define 
\[
V_i:=\{\alpha\in K\mid \alpha t^i\in A\}  
\]
the $k$-vector subspace of $K$ which is a coefficient of $t^i$ in $A$. 
We have $V_i= K$ for all sufficiently large $i$ because $B/A$ is a graded $A$-module of finite length. 
We set 
\[
c=\min\{i\mid V_j=K \mbox{~for~all~} j\ge i\}. 
\]
For $q=p^e\ge c$, we have 
\begin{equation*}
M^{(e)}_i=
\left\{
	\begin{array}{ll}
	\displaystyle\bigoplus_{j\ge 1} K\cdot t^{j+\frac{i}{q}} & \mbox{ $(V_i=0)$, } \\[7mm]
	\displaystyle\bigoplus_{j\ge 0} K\cdot t^{j+\frac{i}{q}} & \mbox{ $(V_i=K)$, } \\[7mm]
	\displaystyle V_i^{\frac{1}{q}}\cdot t^{\frac{i}{q}}\oplus(\bigoplus_{j\ge 1} K\cdot t^{j+\frac{i}{q}}) & 
\mbox{ $(0\subsetneq V_i \subsetneq K)$.} 
	\end{array}
\right. 
\end{equation*}
It is easy to see that $M^{(e)}_i\cong B$ if $V_1=0$ or $K$. 
Note that $V_i^{1/q}=\{\alpha^{1/q}\mid \alpha\in V_i\}$ is also a $k$-vector subspace of $K$ since $K$ is a perfect field. 
\begin{Lemma}\label{isob}
Let the notation be as above. 
Let $q_1=p^{e_1}$, $q_2=p^{e_2}\ge c$ and $i_1,i_2\ge 0$ such that $0\subsetneq V_{i_1}, V_{i_2}\subsetneq K$. 
Then $M^{(e_1)}_{i_1}$ is isomorphic to $M^{(e_2)}_{i_2}$ as graded module up to shift of grading if and only if 
$\beta V^{1/q_1}_{i_1}=V^{1/q_2}_{i_2}$ for some $\beta\in K^*=K\backslash \{0\}$. 
\end{Lemma}
\begin{proof}
A graded homomorphism $\phi :M^{(e_1)}_{i_1}\to M^{(e_2)}_{i_2}$ can be identified with some homogeneous element of $B$: 
\[
\underline{\Hom}_A(M^{(e_1)}_{i_1}, M^{(e_2)}_{i_2})\hookrightarrow  \underline{\Hom}_C(C\cdot t^\frac{i_1}{q_1}, C\cdot t^\frac{i_2}{q_2})
\cong C\Bigl(\frac{i_1}{q_1}-\frac{i_2}{q_2}\Bigr), 
\]
where $C=B[t^{-1}]=K[t,t^{-1}]$. 
Let $\phi\in \Hom_A(M^{(e_1)}_{i_1}, M^{(e_2)}_{i_2})$ be a homogeneous homomorphism which maps to a homogeneous element 
$\beta t^n\in C\Bigl(\frac{i_1}{q_1}-\frac{i_2}{q_2}\Bigr)$ under the above inclusion. 
Then for $g\cdot t^{i_1/q}\in M^{(e_1)}_{i_1}\subset Bt^{i_1/q}$ with $g\in B$, $\phi(g\cdot t^{i_1/q})=\beta g\cdot t^{n+i_2/q}$. 
Hence $n\ge 0$ and $\phi$ is an isomorphism if and only if $n=0$ and $\beta V^{1/q_1}_{i_1}=V^{1/q_2}_{i_2}$. 
Hence there is one-to-one correspondence between the set of graded 
isomorphisms from $M^{(e_1)}_{i_1}$ to $M^{(e_2)}_{i_2}$ and the set $\{\beta\in K^* \mid \beta V^{1/q_1}_{i_1}=V^{1/q_2}_{i_2}\}$. 

\end{proof}
We will present examples of one-dimensional domain which does not have finite (graded) F-representation type. 
\begin{Example}\upshape
Let $k=\bigcup_{e\ge 1}\F_2(u^{1/2^e})$ be the perfect closure of a rational function field $\F_2(u)$. 
Let $A=k[{x},{y}]/({x}^4+{x}^2{y}^2+u{x}{y}^3+{y}^4)$, $\deg {x}=\deg {y}=1$, 
and $\hat{A}=k[[{x},{y}]]/({x}^4+{x}^2{y}^2+u{x}{y}^3+{y}^4)$. 
Since ${x}^4+{x}^2+v^q{x}+1$ is a reduced polynomial in $\F_2[v,{x}]$ for all $q=2^e$, 
it follows that ${x}^4+{x}^2{y}^2+u{x}{y}^3+{y}^4$ is a reduced polynomial in $k[{x},{y}]$. 
We will prove that $A$ does not have finite graded F-representation type, 
and $\hat{A}$ does not have finite F-representation type. 

Let $\alpha\in \overline{k}$ be a root of the irreducible polynomial ${x}^4+{x}^2+u{x}+1$, and set $K=k(\alpha)$. 
Then $A\cong k[\alpha t,t]\subset K[t]$ and the integral closure $B$ of $A$ is isomorphic to $K[t]$, a polynomial ring over $K$. 
Note that $0\subsetneq V_i\subsetneq K$ if and only if $0\le i \le 2$, and $V_i=\bigoplus _{j=0}^{i}k\alpha^j$ for $0\le i \le 2$ 
and $V_i=K$ for all $i\ge 3$. 

We will show that $M^{(e_2)}_1\not\cong M^{(e_1)}_1$ for any $e_2>e_1\ge 2$. 
Assume, to the contrary, that $M^{(e_2)}_1\cong M^{(e_1)}_1$ for some $e_2>e_1\ge 2$. 
We set $q_i=2^{e_i}$, and $e=e_2-e_1$. 
Then there exists $\beta\in K^*$ such that $\beta V_1^{1/q_1}=V_1^{1/q_2}$ by the previous lemma. 
Since $V_1^{1/q}=k\oplus k\alpha^{1{/q}}$, 
there exist $a,b,c,d\in k$ such that 
\[
\beta=a+b\alpha^{1/q_2},~~ \beta\alpha^{1/q_1}=c+d\alpha^{1/q_2}. 
\]
It follows that 
\[
b^{q_2}\alpha^{2^{e}+1}+a^{q_2}\alpha^{2^{e}}-d^{q_2}\alpha-c^{q_2}=0. 
\]
We will show that $1,\alpha,\alpha^{2^e},\alpha^{2^e+1}$ are linearly independent over $k$ for any $e\ge 1$. 
If this is proved, then $a=b=c=d=0$ which contradicts that $\beta\neq 0$. 
We claim that or $e\ge 2$ there exist polynomials $f_e,g_e,h_e\in k[u]$ such that $f_e\neq 0$, $g_e\neq 0$,  
\[
\alpha^{2^e}=
	f_{e}\alpha^2 +g_{e}\alpha+h_{e}
\]
and $\deg_u f_e=\deg_u g_e-1$ if $e$ is even and $\deg_u f_e=\deg_u g_e+1$ if $e$ is odd. 
We prove this claim by induction on $e$. 
If $e=2$, then $\alpha^4={\alpha}^2+u{\alpha}+1$ and thus $f_2=1$ and $g_2=u$. 
If the claim holds true for $e$, 
then 
\[
\alpha^{2^{e+1}}=f_{e}^2\alpha^4 +g_{e}^2\alpha^2+h_{e}^2=f_{e}^2({\alpha}^2+u{\alpha}+1)+g_{e}^2\alpha^2+h_{e}^2
={(f_e^2+g_e^2)}\alpha^2+uf_e^2\alpha+f_e^2+h_e^2, 
\]
and thus $f_{e+1}=f_e^2+g_e^2$ and $g_{e+1}=uf_e^2$. Note that $f_{e+1}\neq 0$ as $\deg_u f_e\neq \deg_u g_e$. 
Since $\deg_u f_{e+1}=2\max\{\deg_u f_e, \deg_u g_e\}$ and $\deg{g_{e+1}}=2\deg_u f_e+1$, 
the claim also holds true for $e+1$. By induction, the claim is true for every $e\ge 2$. 
Hence it follows that $1,\alpha,\alpha^{2^e},\alpha^{2^e+1}$ are linearly independent over $k$ for all $e\ge 1$ 
as $1,~\alpha,~\alpha^2,~\alpha^3$ are linearly independent over $k$. 
Therefore $M^{(e_2)}_1\not\cong M^{(e_1)}_1$ for all $e_2>e_1\ge 3$. 
Hence $A$ does not have finite graded F-representation type. 

We will prove that $\hat{A}$ does not have finite F-representation type. 
Let $\hat{B}$ be the integral closure of $\hat{A}$. 
Note that $B\cong B\otimes_A \hat{A}=K[[t]]$ and $\hat{A}^{1/q}\cong \bigoplus_{i=0}^{q-1} \hat{M}^{(e)}_i$ 
where $\hat{M}^{(e)}_i=M^{(e)}_i\otimes \hat{A}$. 
Assume that there is an isomorphism $\phi:\hat{M}^{(e_2)}_1\to \hat{M}^{(e_1)}_1$ for some $e_2>e_1\ge 3$. 
Let $\hat{\Psi}$ be the inclusion 
$\Hom_{\hat{A}}(\hat{M}^{(e_2)}_1,\hat{M}^{(e_1)}_1)\hookrightarrow  \Hom_{\hat{B}}(\hat{B}t^{1/q_1},\hat{B}t^{1/q_2})\cong \hat{B}$, 
and $\hat{\Psi}(\phi)=\sum_{j\ge 0} \beta_j t^j\in \hat{B}$ $(\beta_j\in K)$. 
Since $\phi$ is an isomorphism, it follows that that $\beta_0\neq 0$ and $\beta_0 V_1^{1/q_1}=V_1^{1/q_2}$, which is a contradiction. 
\end{Example}
\begin{Example}\upshape
Let $k=\bigcup_{e\ge 1}\F_2(u^{1/2^e})$, 
$A=k[{x},{y}]/({x}^6+{x}{y}^5+u{y}^6)$, $\deg {x}=\deg {y}=1$, and $\hat{A}=k[[{x},{y}]]/({x}^6+{x}{y}^5+u{y}^6)$. 
Then $A$ does not have finite graded F-representation type, 
and $\hat{A}$ does not have finite F-representation type. 
\end{Example}
If $k$ is a finite field, then we can prove that $A$ has finite F-representation type. 
\begin{Theorem}\label{one dim ff}
Let $A$ be a one dimensional complete local or $\N$-graded domain of prime characteristic $p$. 
If $k$ a finite field, 
then $A$ has finite F-representation type. 
\end{Theorem}
\begin{proof}
In the case where $A$ is an $\N$-graded ring, 
since $\{V_i^{1/q}\mid q=p^e,~i\ge 0\}$ is a finite set, we have the assertion by Lemma \ref{isob}. 

Assume that $A=(A,\m,k)$ is a one-dimensional complete local domain. 
Let $B=K[[t]]$ be a normalization of $A$, and set $D=k+tB=k[[\alpha t\mid \alpha\in K]]$. 
For $\sum_{i\ge n}\beta_i t^i\in B$, $\beta_n\neq 0$, we set $\initial_B(f)=\beta_nt^n$. 
For $i\in\N$, let 
\[
V_i=\{\beta\mid \initial_B(f)=\beta t^i \mbox{~~for~~some~~}f\in A\}. 
\]
Since $\dim_kB/A<\infty$, it follows that $V_i=K$ for all sufficiently large $i$. 
We set 
\[
c=\min\{i\mid V_j=K \mbox{~~for~~all~~}j\ge i\}. 
\]
Since $A$ is complete, we have $\beta t^n\in A$ for all $\beta\in K$ and $n\ge c$. 
Therefore, $D^q\subset A$ for all $q=p^e\ge c$, and thus $A\subset D\subset A^{1/q}$. 
In particular, $A^{1/q}$ is a $D$-module. 
For $i$ with $V_i\neq 0$, there exists finite set $G_i\subset A$ satisfying following properties: 
\begin{enumerate}
\item $\{\initial_B(g)\mid g\in G_i\}$ is a $k$-basis of $V_it^i=\{\beta t^i\mid \beta\in V_i\}$. 
\item For all $g\in G_i$, $g$ has a form $g=\beta_it^i+\sum_{j> i}\beta_jt^j$, $\beta_j\not\in V_j$. 
\end{enumerate}
Note that if $i\ge c$ then $G_i=\{\alpha_1t^i,\dots,\alpha_rt^i\}$ where $r=[K:k]$ and $\alpha_1,\dots,\alpha_r$ is a $k$-basis of $K$. 
It is easy to prove that $\bigcup_{i=0}^{q-1}G_i^{1/q}$ is a system of generators of $A^{1/q}$ as a $D$-module for $q\ge c$. 
Let 
\begin{eqnarray*}
N^{(e)}&=&D\cdot (\bigcup_{i=0}^{c-1}G_i^{1/q})\\
M^{(e)}_i&=&D\cdot G_i^{1/q}~~~~~\mbox{~~for~~} c\le 0\le q-1. 
\end{eqnarray*}
Then 
\[
A^{1/q}=(\bigoplus_{i=c}^{q-1}M_i^{(e)})\oplus N^{(e)}, 
\]
and $M^{(e)}_i\cong B$ for all $c\le i\le q-1$. 
Assume that $\#K=p^f$. 
To complete the proof, we prove that $N^{(e_1)}\cong N^{(e_2)}$ for $e_1,e_2\in \N$ with $q_1=p^{e_1}, q_2=p^{e_2}\ge c$, 
and $e_1\equiv e_2 \mod f$. 
Let $\varphi:\bigoplus_{i=0}^{c-1}Bt^{i/q_1}\to \bigoplus_{i=0}^{c-1}Bt^{i/q_2}$, 
$t^{i/q_1}\mapsto t^{i/q_2}$ be an isomorphism of free $B$-modules (and hence an isomorphism as $D$-modules). 
Note that $N^{(e_j)}$ is a $D$-submodule of $\bigoplus_{i=0}^{c-1}Bt^{i/q_j}$ for $j=1,2$ by the definition of $G_i$. 
Since $\beta^{p^f}=\beta$ for $\beta\in K$, 
$\varphi(g^{1/q_1})=g^{1/q_2}$ for $g=\sum_{i=0}^{c-1}\beta_it^i\in B$ if $e_1\equiv e_2 \mod f$. 
Therefore $\varphi$ induces one-to-one correspondence between 
$\bigcup_{i=0}^{c-1}G_i^{1/q_1}$ and $\bigcup_{i=0}^{c-1}G_i^{1/q_2}$ if $e_1\equiv e_2 \mod f$. 
This implies that if $e_1\equiv e_2 \mod f$, 
the restriction of $\varphi$ to $N^{(e_1)}$ is an isomorphism form $N^{(e_1)}$ to $N^{(e_2)}$ as $D$-modules, 
and thus as $A$-modules. 
Therefore, $A$ has finite F-representation type. 
\end{proof}
\section{Higher dimensional case}
We end this paper with a few observations on higher dimension rings of finite F-representation type. 
Let $k$ be a field of positive characteristic $p$ with $[k:k^p]<\infty$. 
We begin with the question posed by Brenner. 
\begin{Question}[Brenner]\upshape
Does the ring $k[{x}, {y}, {z}]/({x}^2 + {y} ^3 + {z}^7)$ have finite F-representation type?
\end{Question}
\begin{Observation}\upshape
Let ${S}$ be an F-finite Cohen-Macaulay local (resp. graded) ring of finite (resp. graded) Cohen-Macaulay type, 
and $R$ a local ring such that ${S}\subset R \subset {S}^{1/q'}$ for some $q'=p^{e'}$. 
Let $M$ be an $R$-module (resp. a graded $R$-module). 
Since $({S}^{1/q'})^{q}\subset R$ for $q\ge q'$, 
$^{e}\!M$ has an ${S}^{1/q'}$-module structure for $e\ge e'$. 
If $M$ is a maximal Cohen-Macaulay $R$-module, then $^{e}\!M$ is a maximal Cohen-Macaulay ${S}^{1/q'}$-module, 
and thus $M$ has finite F-representation type. 
In particular, if $R$ is Cohen-Macaulay, then $R$ has finite F-representation type. 
\end{Observation}
\begin{Example}\upshape
Let $R=k[s^q,st,t]\cong k[x,y,z]/(y^q-xz^q)$. 
Since $k[s^q,t^q]\subset R \subset (k[s^q,t^q])^{1/q}$, $R$ has finite F-representation type. 
\end{Example}
\begin{Example}\upshape
Let ${S}$ be an F-finite regular local ring  (resp. a polynomial ring over a field), 
and let $f\in S$ be an element (resp. a homogeneous element), and $R=S[f^{1/q}]$. 
Then $R$ has finite F-representation type. 
In particular, $k[{x}, {y}, {z}]/({x}^2 + {y} ^3 + {z}^7)$ has finite F-representation type if $p=2,3$, or $7$. 
\end{Example}
We can prove a little more general theorem. 
\begin{Theorem}
Let $R$ be an F-pure complete local (resp. graded) domain of finite F-representation type,  
$e_1,\dots,e_r$ positive integers, and $q_i=p^{e_i}$. 
Let $f_1,\dots,f_r$ be (resp. homogeneous) elements of $R$, and  
\[
S=R[{x}_1,\dots,{x}_r]/({x}_1^{q_1}+f_1,\dots,{x}_r^{q_r}+f_r). 
\]
Then $S$ has finite (resp. graded) F-representation type. 
\end{Theorem}
\begin{proof}
Note that if $R$ is a graded ring, then $S$ is also a graded ring by assigning $\deg({x}_i)=\deg(f_i)/q_i$. 

Let $\tilde{e}=\max\{e_1,\dots,e_r\}$ and $\tilde{q}=p^{\tilde{e}}$. 
First, we prove the theorem in the case of $f_i=0$ for all $i$. 
Since $S=R[{x}_1,\dots,{x}_r]/({x}_1^{q_1},\dots,{x}_r^{q_r})$ is a free $R$-module of finite rank, 
$S$ has finite F-representation type as an $R$-module. 
On the other hand, since $({x}_1,\dots,{x}_r)\cdot{}^eS=({x}_1,\dots,{x}_r)^{[q]}S=0$ for $e\ge \tilde{e}$, 
a decomposition of ${}^eS$ as an $R$-module can be regarded a decomposition as an $S$-module. 
Hence $S$ has finite F-representation type. 

In the general case, since $R$ is F-pure, $R$ is a direct summand of $R^{1/\tilde{q}}$, 
and thus $R[{x}_1,\dots,{x}_r]$ is a direct summand of $R^{1/\tilde{q}}[{x}_1,\dots,{x}_r]$. 
Hence $S$ is a direct summand of 
\begin{eqnarray*}
&&R^{1/\tilde{q}}[{x}_1,\dots,{x}_r]/({x}_1^{q_1}+f_1,\dots,{x}_r^{q_r}+f_r)\\
&=&R^{1/\tilde{q}}[{x}_1,\dots,{x}_r]/(({x}_1+f_1^{1/q_1})^{q_1},\dots,({x}_r+f_r^{1/q_r})^{q_r})\\
&\cong& R^{1/\tilde{q}}[{x}_1,\dots,{x}_r]/({x}_1^{q_1},\dots,{x}_r^{q_r}). 
\end{eqnarray*}
Since $R^{1/\tilde{q}}$ has finite F-representation type, $R^{1/\tilde{q}}[{x}_1,\dots,{x}_r]/({x}_1^{q_1},\dots,{x}_r^{q_r})$ has 
finite F-representation type as proved above. 
Therefore $S$ has finite F-representation type by Example \ref{exffrt} (v) and (vi). 
\end{proof}

\end{document}